\documentclass[mlmain]{jmlr}

\usepackage{longtable}

\usepackage{booktabs}
\usepackage[load-configurations=version-1]{siunitx} 

\usepackage{tikz}
\usepackage{tikz-cd}

\newcommand{\R}{\mathbb R}

\newcommand{\curC}{\mathcal C}
\newcommand{\curD}{\mathcal D}
\newcommand{\curE}{\mathcal E}
\newcommand{\curI}{\mathcal I}

\newcommand{\curU}{\mathcal U}

\newcommand{\Tk}{\textnormal{Tk}}
\newcommand{\code}{\textnormal{code}}
\newcommand{\Code}{\mathbf{Code}}
\newcommand{\nring}{\mathbf{NRing}}
\newcommand{\om}{\mathbf{OM}}
\newcommand{\omring}{\mathbf{OMRing}}
\newcommand{\pcode}{\mathbf{P}_\Code}

\theorembodyfont{\upshape}
\theoremheaderfont{\scshape}
\theorempostheader{:}
\theoremsep{\newline}

\jmlrvolume{}
\firstpageno{1}
\editors{List of editors' names}

\jmlryear{2025}
\jmlrworkshop{Symmetry and Geometry in Neural Representations}

\title[Covering Relations in $\pcode$]{Covering Relations in the Poset of\\ Combinatorial Neural Codes
}
 \author{\Name{R. Amzi Jeffs} \Email{amzi.jeffs@pnnl.gov}\\
 \addr Pacific Northwest National Laboratory, Richland, Washington.
 \AND
 \Name{Trong-Thuc Trang} \Email{ttrang2019@fau.edu}\\
 \addr Florida Atlantic University, Boca Raton, Florida.
 }

\begin{document}

\maketitle

\begin{abstract}
A combinatorial neural code is a subset of the power set $2^{[n]}$ on $[n]=\{1,\dots, n\}$, in which each $1\leq i\leq n$ represents a neuron and each element (codeword) represents the co-firing event of some neurons. Consider a space $X\subseteq\mathbb{R}^d$, simulating an animal's environment, and a collection $\mathcal{U}=\{U_1,…,U_n\}$ of open subsets of $X$. Each $U_i\subseteq X$ simulates a place field which is a specific region where a place cell $i$ is active. Then, the code of $\mathcal{U}$ in $X$ is defined as
$\mbox{code}(\mathcal{U},X)=\left\{\sigma\subseteq[n]\bigg|\bigcap_{i\in\sigma} U_i\setminus\bigcup_{j\notin\sigma}U_j\neq\varnothing\right\}$. If a neural code $\mathcal{C}=\mbox{code}(\mathcal{U},X)$ for some $X$ and $\mathcal{U}$, we say $\mathcal{C}$ has a realization of open subsets of some space $X$. Although every combinatorial neural code obviously has a realization by some open subsets, determining whether it has a realization by some open convex subsets remains unsolved. Many studies attempted to tackle this decision problem, but only partial results were achieved. In fact, a previous study showed that the decision problem of convex neural codes is NP-hard. Furthermore, the authors of this study conjectured that every convex neural code can be realized as a minor of a neural code arising from a representable oriented matroid, which can lead to an equivalence between convex and polytope convex neural codes. Even though this conjecture has been confirmed in dimension two, its validity in higher dimensions is still unknown. To advance the investigation of this conjecture, we provide a complete characterization of the covering relations within the poset $\pcode$ of neural codes.
\end{abstract}

\section{Introduction}
\label{sec:intro}
In 1971, John O'Keefe and Jonathan Dostrovsky discovered individual hippocampal neurons in rats that were active when the rats were in specific locations \citep{od71}. These neurons are known as \emph{place cells}, and the regions in the environment where these neurons fire are called \emph{place fields}. For long, these neurons have been thought of as encoding a cognitive map of a mammal's environment. Many mathematical tools such as algebraic topology have been used to decode any meaningful representation of neural signals coming from place cells. Curto and Itskov applied topological data analysis to place cell data in \citep{ci08} and found that many topological features of the stimulus space can be extracted from cell groups when place fields were assumed to be convex and even when a small portion of them was multi-peaked. Every neural code obtained from cell groups also turns out to carry one geometric information of place fields, their convexity \citep{civy13}, \citep{cur17}. This encoded information is intrinsic because it can be inferred from the neural signals of place cells alone. To have a clearer picture, we provide the following definitions and some conventions made for this work.

Place fields that correspond to active place cells in a mammalian hippocampus can be topologically simulated by open subsets of some ambient topological space representing the environment that the animal explores. Then, a code can be naturally derived from such realization of open subsets as follows. Note that, throughout this work, any ambient space is restricted to some Euclidean space $\R^d$.
\begin{definition}[Codes of covers]
\label{def:codes-of-covers}
Given a cover $\curU=\{U_1,\dots,U_n\}$ of open subsets of some Euclidean space $\R^d$. The \textbf{code of $\curU$ in $\R^d$} is defined as
\[
\code(\curU,\R^d)=\left\{\sigma\subseteq[n]\bigg|\bigcap_{i\in\sigma}U_i\setminus\bigcup_{j\not\in\sigma}U_j\neq\varnothing\right\}.
\]
\end{definition}
Conventionally, if $\sigma=\varnothing$, then the intersection of $U_i$ over all $i\in\sigma$ is $\R^d$. The cover $\mathcal{U}$ is called a \emph{realization} of $\mbox{code}(\mathcal{U},\R^d)$. The definition of combinatorial neural codes is given below. Some past works define this mathematical structure using binary patterns, but they are equivalent.

\begin{definition}[Combinatorial neural codes]
\label{def:combinatorial-codes}
A \textbf{(combinatorial) neural code} $\curC$ is a collection of some subsets of $[n]=\{1,\dots,n\}$ for some natural number $n$. Elements of $[n]$ represent \textbf{neurons} and elements of $\curC$ called \textbf{codewords} represent the co-firing events of some neurons.
\end{definition}
In this work, every combinatorial neural code must contain the $\varnothing$-codeword\footnote{In reality, this means that some certain region of the environment is not covered by any place field.} denoted by $\emptyset$. Without any ambiguity, instead of writing a codeword as a subset of $[n]$ (e.g., $\{1, 2\}$) we usually write it compactly as a string of numbers in ascending numerical order (e.g., $12$). Hereafter, the term \emph{neural code} or, shortly, \emph{code} will be used interchangeably to replace the full term \emph{combinatorial neural code}. 


Obviously, every neural code $\curC=\mbox{code}(\mathcal{U},\R^d)$ for some $\R^d$ and $\mathcal{U}$, and we say that $\curC$ has a realization. Furthermore, if there is a realization $\mathcal{U}$ consisting of open convex subsets of $\R^d$, we say that $\curC$ is an \emph{open convex neural code}. Many partial attempts have been made to answer the question ``When is a neural code convex?". The list includes, but is not limited to \citep{civy13}, \citep{gi14}, \citep{cy15}, \citep{cgjmorsy15}, \citep{cgib16}, \citep{gny16}, \citep{lsw17}, \citep{mt17}, \citep{joy18}, \citep{gms19}, and \citep{cfs19}. To advance the study of open convex neural codes, Jeffs introduced in \citep{morphisms20} the poset of combinatorial neural codes $\pcode$ and connected the category of neural codes $\Code$ with the category of neural rings $\nring$ studied in previous research. Continuing on this, Kunin, Lienkaemper, and Rosen built connections of these two categories with the categories of oriented matroids $\om$ and oriented matroid rings $\omring$ in \citep{klr20}. 
    \begin{equation*}
        \begin{tikzcd}
            \mathbf{OM} \dar[swap]{\mathsf{W^+}} \rar{\mathsf{S}} & \mathbf{OMRing} \dar{\mathsf{D}}\\
            \mathbf{Code} \rar[swap]{\mathsf{R}} & \mathbf{NRing}
        \end{tikzcd}
    \end{equation*}
These connections gave them a nice way of showing that the decision problem of open convex neural codes is in fact NP-hard and that open convex polytope neural codes (a subclass of open convex neural codes) are essentially those that lie below some representable oriented matroids in $\pcode$. They also conjectured that open convex neural codes are open convex polytope.

\begin{conjecture}[\citealp{klr20}]
\label{conj:convex-char}
A neural code $\curC$ is convex if and only if $\curC$ lies below some representable oriented matroid in $\pcode$.
\end{conjecture}

The above conjecture is true in the plane $\R^2$ \citep{bj23}, but has not been proven in general. If this conjecture holds, we can avoid the decision problem of open convex neural codes and instead check the representability of all oriented matroid codes lying above a given code. While being $\exists\R$-hard\footnote{The complexity class $\exists\R$ is read as the \emph{existential theory of the reals}.}, the representability of oriented matroids is still an active area of research and can be solved for small ranks. However, to be certain of this replacement, we need to validate or refute \conjectureref{conj:convex-char}. One way to do this is to construct a method to travel up the poset $\pcode$ towards oriented matroid codes above it. Although, for the time being, we have not been able to fully construct a method for this idea, we know how one can ``climb up" from one code to any code that covers it, which will be presented in this paper. Our paper is structured as follows: some background of the poset $\pcode$ is given in \sectionref{sec:poset}, and the reader can find our main contribution which is the method of traveling upward $\pcode$ in \sectionref{sec:up-covering}. All supporting results can be found in \appendixref{sec:down-covering} and \appendixref{sec:int-comp-up-covering}.

\section{The Poset of Combinatorial Neural Codes}
\label{sec:poset}
In the following, we provide some background of the poset $\pcode$ of combinatorial neural codes. The reader may also read \appendixref{sec:down-covering} for necessary information.

\begin{definition}[Trunks in neural codes]
\label{def:Tks-in-codes}
A \textbf{trunk} in a neural code $\curC$ is the set $\Tk_\curC(\sigma)$ (possibly empty) defined by
\[
\Tk_\curC(\sigma) = \{\tau\in\curC\mid \sigma\subseteq \tau\}.
\]
for some $\sigma\subseteq[n]$.
\end{definition}

\begin{definition}[Simple trunks and proper trunks]
\label{def:simple-Tks}
A trunk in a neural code $\curC$ of a single neuron $i$, $\Tk_\curC(\{i\})$, is called a \textbf{simple trunk} and simply denoted by $\Tk_\curC(i)$. A trunk in $\curC$ is \textbf{proper} if it is nonempty and not the trunk of $\emptyset$ (i.e., not $\curC$).
\end{definition}

An obvious property of trunks is that the intersection between any two trunks is again a trunk \cite[Proposition 2.2]{morphisms20}. The purpose of the trunks in a neural code is to encode important combinatorial properties of certain groups of codewords, and the mappings which preserve these important combinatorial relations among the codewords in a neural code are called morphisms or neural codes. Intuitively, the notion of trunks in neural codes is analogous to that of open stars in simplicial complexes. Morphisms of neural codes is then in analogy to that of continuous functions between topological spaces: they are ``continuous" with respect to trunks.

\begin{definition}[Morphisms of neural codes]
\label{def:morphisms-of-codes}
A function $f:\curC\to \curD$ is a \textbf{morphism} of neural codes if the preimage $f^{-1}(T)$ of every proper trunk $T$ in $\curD$ is a proper trunk in $\curC$.
\end{definition}

This definition uses proper trunks to define code morphisms rather than just any trunks as in \citep{morphisms20}. As for this adjustment mentioned in \citep[Remark 3.1.2.]{amzithesis}, it ensures several nice properties of neural codes under morphisms. The following, which implies that the image of every neural code in the domain is a subcode of the neural code in the codomain, is an example.

\begin{proposition}
If $f$ is a morphism of neural codes, then $f(\emptyset)=\emptyset$.
\end{proposition}
\begin{proof}
See \cite[Proposition 3.1.3]{amzithesis}.
\end{proof}

\begin{definition}
\label{def:defining-morphisms} 
Given some proper trunks $T_1, \dots, T_m$ in $\curC$. Then, the function $f:\curC\to 2^{[m]}$ given by $c\mapsto\{j\in[m]\mid c\in T_j\}$ is called the \textbf{morphism determined by $T_1,\ldots, T_m$}.
\end{definition}

\begin{proposition}
The function described above in Definition \ref{def:defining-morphisms} is a morphism. Moreover, every morphism arises in this way: if $\curD\subseteq 2^{[m]}$ and $f:\curC\to \curD$ is any morphism, then $f$ can be obtained by restricting to $\curD$ the codomain of the morphism $g$ determined by the trunks $T_j = f^{-1}(\Tk_\curD(j))$ in $\curC$.
\end{proposition}
\begin{proof}
See \citep[Propositions 2.11 and 2.12]{morphisms20} .
\end{proof}

\begin{remark}
Sometimes, it will be convenient to talk about the morphism determined by a set of trunks without indexing the trunks. We may do this when we only care about the isomorphism class of a neural code because a choice of indexing simply corresponds to a permutation of neurons in the codomain, which is an isomorphism. 
\end{remark}

In the following, we present the poset $\pcode$ that was first introduced in \citep{morphisms20} and then slightly modified in \citep{amzithesis}. The formulation below has the advantage that we no longer need to consider ``replacement by a trunk" as a possible covering relation, since every nonempty trunk in $\curC$, up to adding the $\varnothing$-codeword $\emptyset$, is a minor of $\curC$ via a surjective morphism (see Proposition \ref{prop:Tks-as-minors}). The formulation below allows us to maintain the previously desirable properties of $\pcode$ while accounting for the usual neural code convention that $\emptyset$ is always a codeword.

\begin{definition}[The poset $\pcode$]
Let $\curC$ and $\curD$ be neural codes. We say that $\curD$ \textbf{is a minor of} $\curC$ (written $\curD\leq\curC$) if there is a surjective morphism $f:\curC\to \curD$. As a consequence, this relation defines a partial order on the set of distinct isomorphism classes of neural codes, denoted by $\pcode$.
\end{definition}

Note that the above definition no longer allows ``replace $\curC$ by a trunk" as a possible covering relation as in \citep{morphisms20} since we require the $\varnothing$-codeword to be in every code. However, the following proposition shows that with a small adjustment we do not lose too much when doing this.

\begin{proposition}\label{prop:Tks-as-minors} Let $T$ be a (possibly empty) trunk in a neural code $\curC$. Then, $\{\emptyset\}\cup T\leq\curC$.
\end{proposition}
\begin{proof}
Consider the function sending $T\subset\curC$ to itself and all other codewords in $\curC$ to $\emptyset$. This function from $\curC$ to $\{\emptyset\}\cup T$ is obviously surjective and a morphism of neural codes because the preimage of any proper trunk $\Tk_T(\sigma)$ in $T$ is equal to $\Tk_\curC(\sigma)\cap T$ which is a proper trunk in $\curC$. 
\end{proof}

\section{Upward Covering Relations in $\pcode$}
\label{sec:up-covering}
In this section, we provide a method to describe all codes that cover a given neural code $\curD$ in $\pcode$ as opposed to \appendixref{sec:down-covering}. To fully understand the development of how we construct covering codes, the reader may read \appendixref{sec:int-comp-up-covering} before this section. In \appendixref{sec:int-comp-up-covering}, we instead described $\curC$ in terms of $\curD$ when they both are intersection-complete. In the following, we extend the idea in \appendixref{sec:int-comp-up-covering} to all neural codes. We start by showing that every covering relation $\curC$ and $\curD$ in $\pcode$ extends in a natural way to a covering relation between their intersection-completions. The hat notation (e.g. $\widehat\curC$) will be used to denote the intersection-completion of any code (e.g. $\curC$). This extension will be a key tool in characterizing all upward covering relations in $\pcode$. 

\begin{lemma}\label{lem:unique-extension}
Let $f:\curC\to \curD$ be a morphism. There exists a unique morphism $g:\widehat \curC\to\widehat \curD$ such that the diagram
\[
\begin{tikzcd}
\curC \arrow[r, hook]\arrow[d,"f",swap]& \widehat \curC\arrow[d,dashed,"g"]\\
\curD \arrow[r,hook] & \widehat \curD
\end{tikzcd}
\]
commutes. Moreover, if $f$ is surjective, then so is $g$.
\end{lemma}
\begin{proof}
First, let us prove that the morphism $g$ exists. We know that $f$ is determined by some collection of trunks $\{T_1,\ldots, T_m\}$ in $\curC$. Each $T_j$ can be naturally associated with the smallest trunk (in terms of inclusion) $S_j$ in $\widehat \curC$ that contains it by setting $S_j$ to be the intersection of all trunks containing $T_j$ in $\widehat\curC$. The collection $\{S_1,\ldots, S_m\}$ in $\widehat \curC$ determines a morphism $h:\widehat \curC\to 2^{[m]}$. Observe that if $c\in \curC$, then $f(c) = h(c)$, so $h$ extends the morphism $f$ to $\widehat \curC$ while also extending the codomain to $2^{[m]}$. By Proposition \ref{prop:image-of-intersection} (\appendixref{sec:int-comp-up-covering}), the image of $\widehat \curC$ under $h$ is a subset of $\widehat \curD$. Thus, we can restrict $h$ to a morphism $g:\widehat\curC\to\widehat\curD$. Moreover, Proposition \ref{prop:image-of-intersection} (\appendixref{sec:int-comp-up-covering}) implies that the action of $g$ on a codeword in $\widehat \curC$ is uniquely determined by its action on $\curC$, and so the choice of $g$ is unique.

For the surjectivity statement, suppose that $f$ is surjective and consider an arbitrary codeword in $\widehat \curD$. Since $f$ is surjective, this codeword is an intersection of various $f(c)$. This implies that it is the intersection of various $g(c)$, and by Proposition \ref{prop:image-of-intersection} (\appendixref{sec:int-comp-up-covering}), this means that it is the image under $g$ of some codeword in $\widehat \curC$, so $g$ is surjective as desired.
\end{proof}

Such a diagram can help us recognize when codes and their intersection completions cover one another, as described in the lemma below.

\begin{lemma}\label{lem:extension-pullback}
Let $f:\curC\to \curD$ be a surjective morphism, and $g:\widehat \curC\to \widehat \curD$ be the unique surjective morphism so that the diagram \[
\begin{tikzcd}
\curC \arrow[r, hook]\arrow[d,twoheadrightarrow,"f",swap]& \widehat \curC\arrow[d,twoheadrightarrow,"g"]\\
\curD \arrow[r,hook] & \widehat \curD
\end{tikzcd}
\]
commutes. Then $\widehat \curC$ covers $\widehat \curD$ if and only if $\curC$ covers $\curD$. 
\end{lemma}
\begin{proof}
This follows from an application of Proposition \ref{prop:trunks-elems-bijection} (\appendixref{sec:int-comp-up-covering}) and \corollaryref{cor:covering-criterion} (\appendixref{sec:down-covering}). Observe that nonempty trunks in $\curC$ are in bijection with elements of $\widehat \curC$ via the map which sends a nonempty trunk to the intersection of all its elements. Thus the following statements are equivalent:\begin{itemize}
\item[(i)] $\widehat \curC$ covers $\widehat \curD$,
\item[(ii)] $\widehat\curC$ has one more nonempty trunk than $\widehat\curD$,
\item[(iii)] $\widehat \curC$ has one more element than $\widehat \curD$,
\item[(iv)] $\curC$ has  one more nonempty trunk than $\curD$, and
\item[(v)] $\curC$ covers $\curD$. 
\end{itemize}
\end{proof}

\begin{theorem}\label{thm:from-below-abstract}
Let $\curC$ and $\curD$ be codes and suppose that $\curC$ covers $\curD$. Then $\widehat \curC$ is isomorphic to $\widehat \curD_{[\curI]}$ (see Definition \ref{def:int-comp-covering}, \appendixref{sec:int-comp-up-covering}) for some choice of isolated subset $\curI\subseteq \widehat\curD$.
\end{theorem}

\begin{proof}
If $\curC$ covers $\curD$, we may construct the diagram in \lemmaref{lem:extension-pullback}, and conclude that $\widehat\curC$ covers $\widehat \curD$. The rest of the result follows from \theoremref{thm:int-comp-covering} (\appendixref{sec:int-comp-up-covering}).
\end{proof}

\begin{corollary}
Suppose that $\curC$ covers $\curD$. Then $\widehat \curC$ covers $\widehat \curD$.
\end{corollary}


The following definitions provide an important building block of constructing covering codes and a list of all codes that can cover a given code $\curD$, respectively. We will prove the latter is true in Theorem \ref{thm:covering-code-char}.

\begin{definition}\label{def:isolated-subcode}
Let $\mathcal{C}$ be an intersection-complete code, $\mathcal{I}\subseteq\mathcal{C}$ its nonempty intersection-complete subset. Then, $\mathcal{I}$ is \textbf{isolated} if no codeword in $\curC\setminus\curI$ contains any non-minimal codeword in $\curI$; i.e., $\sigma\not\supset\tau$ for every $\sigma\in\mathcal{C}\setminus\mathcal{I}$ and every $\tau\in\mathcal{I}\setminus\{\mu\}$ where $\mu$ is the minimal element of $\mathcal{I}$.
\end{definition}
Note that any singleton set is an isolated subset, and so is any nonempty trunk in intersection-complete codes.


\begin{definition}\label{def:all-covering}
Let $\curD$ be any neural code, and $\curI\subseteq \widehat\curD$ an isolated subset with minimal element $\mu$. We define four types of covering code for $\curD$, subject to certain conditions on $\curI$. We summarize these conditions and the resulting codes in \tableref{tab:covering-code-construction}. The notation $(\cdot)_\alpha$ denotes the action of adding a new neuron $\alpha$ to every codeword in the argument (e.g. $(S)_\alpha=\{c\cup\alpha\,|\,c\in S\}$).
\end{definition}  

\begin{table}[h]
        \centering
        \small       
        \begin{tabular}{|c|l|l|l|}
            \hline
            \textbf{Type} & \textbf{Conditions} & \textbf{Construction} \\
            \hline
            1 & $\mu\in\curD$ & $\curD_{[\curI]}=(\curD\cap\curI)_{\alpha}\cup\curD\setminus\curI\cup\{\mu\}$ \\
            \hline
            2 & $\mu\in\curD$ and $\Tk_\curD(\mu)\setminus\curI\neq\varnothing$ &  $\curD_{(\curI]}=(\curD\cap\curI)_{\alpha}\cup\curD\setminus\curI$  \\
            \hline
            3 &  $\mu\in\curD$ and $\mu=\bigcap\limits_{\sigma\in\curD\cap\curI\setminus\{\mu\}}\sigma$ &  $\curD_{[\curI)}=(\curD\cap\curI\setminus\{\mu\})_{\alpha}\cup\curD\setminus\curI\cup\{\mu\}$ \\
            \hline
            4 & $\mu\notin\curD$, $\mu=\bigcap\limits_{\sigma\in\curD\cap\curI\setminus\{\mu\}}\sigma$, and $\Tk_\curD(\mu)\setminus\curI\neq\varnothing$ & $\curD_{(\curI)}=(\curD\cap\curI)_{\alpha}\cup\curD\setminus\curI$ \\
            \hline
        \end{tabular}
        \caption{Constructions of covering codes, subject to the conditions on an isolated subset.} 
        \label{tab:covering-code-construction}
\end{table}  



Note that the conditions in \tableref{tab:covering-code-construction} are not mutually exclusive. Thus for a given choice of $\curI$ we may be able to form multiple covering codes. Besides, when $\curC$ is intersection-complete, the notation above is compatible with \definitionref{def:int-comp-covering} (\appendixref{sec:int-comp-up-covering}). We give several results below, which eventually show that the codes above are exactly those covering $\curD$ in $\pcode$. Our first result below says that the constructions in Definition \ref{def:all-covering} do not differ too much: when we apply the construction and compute the intersection-completion of the resulting code, we will always obtain the covering code $\widehat\curD_{[\curI]}$ of Definition \ref{def:int-comp-covering} (\appendixref{sec:int-comp-up-covering}).

\begin{theorem}\label{thm:int-completion}
Let $\curD$ be a code, $\curI\subseteq \widehat\curD$ an isolated subset with the minimal codeword $\mu$, and $\curC$ one of the covering codes described in Definition \ref{def:all-covering}. Then $\widehat\curC = \widehat\curD_{[\curI]}$.
\end{theorem}
\begin{proof}
In every type of covering code in \definitionref{def:all-covering}, we see that the new neuron $\alpha$ is only added to elements of $\curD\cap \curI$ (possibly, except for $\mu$). Comparing this to \definitionref{def:int-comp-covering} (\appendixref{sec:int-comp-up-covering}), we see that $\curC\subseteq \widehat\curD_{[\curI]}$. Since $\widehat \curD_{[\curI]}$ is intersection-complete, this immediately implies that $\widehat\curC\subseteq \widehat\curD_{[\curI]}$. 

It remains to prove the reverse inclusion. Let $c\in\widehat \curD_{[\curI]}$. We must show that $c$ is an element or an intersection of elements of $\curC$. First, suppose that $c$ contains $\alpha$ and is not $\mu\cup\alpha$. Then $c\setminus\alpha$ is an element of $\curI$ and equal to an intersection of some codewords in $\curD$. If any of the codewords whose intersection is $c\setminus\alpha$ is not an element of $\curI$, then \definitionref{def:isolated-subcode} implies that $c\setminus\alpha = \mu$, a contradiction. Therefore, $c\setminus\alpha$ is an intersection of codewords in $(\curD\cap \curI)\setminus \{\mu\}$. In every type of covering code, we add $\alpha$ to all such codewords, and hence $c$ is an intersection of some codewords in $\curC$ as desired. Next, suppose that $c$ does not contain $\alpha$ and is not $\mu$. Then $c$ is an intersection of some codewords in $\curD$, at least one of which is not an element of $\curD\cap\curI$. When we form $\curC$, we add $\alpha$ to all codewords in $\curD\cap\curI$ (possibly, except $\mu$). Since at least one of the codewords whose intersection is $c$ is not in $\curD\cap \curI$, the intersection of these codewords in $\curC$ will still be $c$. This case is concluded. This leaves the following two cases: $c = \mu$ and $c = \mu\cup\alpha$. We will show that both $\mu$ and $\mu\cup\alpha$ are codewords in $\widehat\curC$, regardless of covering type $\curC$ is. Each case is considered as follows.
\begin{itemize}
\item[] \textbf{Type 1:} Here the definition stipulates that $\mu$ stays in $\curC$. Since we add $\alpha$ to all elements of $\curD\cap \curI$, we also obtain $\mu\cup\alpha$ as a codeword in $\curC$. 
\item[] \textbf{Type 2:} Since we assume $\mu\in \curD$, this implies $\mu\in \curD\cap\curI$, so $\mu\cup\alpha$ is a codeword in $\curC$. Since $\Tk_\curD(\mu)\setminus \curI$ is nonempty, there is a codeword $c'\in \curC$ which contains $\mu$ but not $\alpha$. This leads to the fact that $c'\cap (\mu\cup\alpha) = \mu$ is a codeword in $\widehat \curC$. 
\item[] \textbf{Type 3:} The definition stipulates that we include $\mu$ as a codeword in $\curC$. We also assume that $\mu$ is the intersection of all elements in $(\curD\cap \curI)\setminus \{\mu\}$, which are exactly the codewords added with $\alpha$. Thus, $\mu\cup\alpha$ is the intersection of all codewords in $\Tk_{\curC}(\alpha)$ and lies in $\widehat\curC$.
\item[] \textbf{Type 4:} Here neither $\mu$ nor $\mu\cup\alpha$ are codewords in $\curC$. However, similarly to Type 2, the condition $\Tk_\curD(\mu)\setminus \curI$ is nonempty will lead to the existence of $\mu$ in $\widehat\curC$. Besides, similarly to Type 3, the assumption that $\mu$ is the intersection of all elements in $(\curD\cap \curI)\setminus \{\mu\}$ will lead to the existence of $\mu\cup\alpha$ in $\widehat\curC$. 
\end{itemize}

We have shown in all cases that every $c\in \widehat \curD_{[\curI]}$ is a codeword in $\widehat\curC$. The result then follows.
\end{proof}

\begin{corollary}\label{cor:all-types-cover}
Every covering code $\curC$ described in \definitionref{def:all-covering} covers $\curD$ in $\pcode$.
\end{corollary}
\begin{proof}
Let $\curC$ be one of the covering codes of \definitionref{def:all-covering}. Note that there is a natural surjective morphism $f:\curC\to \curD$ given by deleting the neuron $\alpha$. By Lemma \ref{lem:unique-extension} we obtain a unique surjective morphism $g:\widehat \curC \to \widehat\curD$, which extends $f$. By Lemma \ref{lem:extension-pullback}, $\curC$ covers $\curD$ if and only if $\widehat \curC$ covers $\widehat \curD$. \theoremref{thm:int-completion} tells us that $\widehat \curC$ is equal to $\widehat\curD_{[\curI]}$, which covers $\widehat\curD$ by \theoremref{thm:int-comp-covering}  (\appendixref{sec:int-comp-up-covering}). Thus, $\curC$ covers $\curD$ as desired.
\end{proof}

To end this, we provide the following theorem which shows that if one uses constructions described in \definitionref{def:all-covering}, they can ``climb up" the poset $\pcode$ from a given code to all its covering codes.

\begin{theorem}\label{thm:covering-code-char}
Let $\curC$ and $\curD$ be codes. The following are equivalent:
\begin{itemize}
\item[(i)]  $\curC$ covers $\curD$, and
\item[(ii)]  $\curC$ is isomorphic to one of the covering codes for $\curD$ described in  \definitionref{def:all-covering}.
\end{itemize}
\end{theorem}
\begin{proof}
\corollaryref{cor:all-types-cover} tells us that (ii) implies (i), and so it remains to prove the converse. Suppose that $\curC$ covers $\curD$ and let $f:\curC\to\curD$ be a surjective morphism with unique surjective extension $g:\widehat \curC \to \widehat \curD$ as guaranteed by \lemmaref{lem:unique-extension}. \lemmaref{lem:extension-pullback} tells us that $\widehat \curC$ covers $\widehat \curD$, and \theoremref{thm:int-comp-covering}  (\appendixref{sec:int-comp-up-covering}) implies we may choose an isolated subset $\curI\subseteq \widehat \curD$ so that $\widehat\curC$ is isomorphic to $\widehat\curD_{[\curI]}$. Since $\pcode$ describes covering relations between isomorphism classes of codes and the diagram is unchanged by an isomorphism on $\widehat \curC$ (which induces an isomorphism on $\curC$), we may replace $\widehat \curC$ by $\widehat \curD_{[\curI]}$. Thus we can assume without loss of generality that $\widehat \curD_{[\curI]}$ is the intersection-completion of $\curC$. We thus have the following diagram:\[
\begin{tikzcd}
\curC \arrow[r, hook]\arrow[d,twoheadrightarrow,"f",swap]& \widehat \curC = \widehat\curD_{[\curI]}\arrow[d,twoheadrightarrow,"g"]\\
\curD \arrow[r,hook] & \widehat \curD
\end{tikzcd}
\] 

Let $\mu$ be the minimal element of $\curI$. Notice that the surjective map $\widehat\curD_{[\curI]} \to \widehat \curD$ is bijective except for the fact that it identifies the codewords $\mu$ and $\mu\cup\alpha$. Thus for the diagram to commute, we see that $\curC$ contains at least the codewords of $\curD$ that do not have a new neuron added to them when forming $\widehat \curD_{[\curI]}$, as well as the codewords obtained by adding $\alpha$ to codewords of $(\curD\cap\curI)\setminus \{\mu\}$. Note that all of these are present in the covering codes described in Definition \ref{def:all-covering}.

This leaves four cases to consider, based on whether $\curC$ contains $\mu$ and/or $\mu\cup\alpha$.

\begin{itemize}
\item \textbf{Case 1:} Suppose $\curC$ contains both $\mu$ and $\mu\cup \alpha$. Then we see that $\curC$ is equal to the construction in Type 1, and indeed $\mu\in \curD$ which is the condition for Type 1.

\item \textbf{Case 2:} Suppose $\curC$ contains $\mu\cup\alpha$ but not $\mu$. In this case $\curC$ is equal to the construction in Type 2. Again we must have $\mu\in \curD$ since this is the image of $\mu\cup \alpha$ under the surjective map $\curC\to \curD$ that deletes $\alpha$. However, for $\widehat \curC$ to be $\widehat \curD_{[\curI]}$, there must exist some codewords in $\curC$ whose intersection is $\mu$. Since $\mu\cup \alpha$ is a codeword of $\curC$ and $\mu$ is not, this is equivalent to the statement that there is a codeword in $\curC$ that properly contains $\mu$ and not $\alpha$. Such a codeword in $\curC$ must come from a codeword in $\Tk_\curD(\mu)\setminus \curI$, and so $\curI$ satisfies the conditions stipulated for Type 2.

\item \textbf{Case 3:} Suppose $\curC$ contains $\mu$ but not $\mu\cup \alpha$. Again $\mu\in \curD$, and we see that $\curC$ is equal to the construction in Type 3. We know that $\mu\cup \alpha$ lies in $\widehat \curC$, and so there must exist some codewords in $\Tk_\curC(\alpha)$ whose intersection is $\mu\cup\alpha$. Codewords in $\Tk_\curC(\alpha)$ are obtained by adding $\alpha$ to elements of $\curD\cap \curI$ except $\mu$, and so the intersection of all codewords in $(\curD\cap \curI)\setminus \{\mu\}$ must be $\mu$, as stipulated for Type 3.

\item \textbf{Case 4:} Lastly, suppose that $\curC$ contains neither $\mu$ nor $\mu\cup \alpha$. Then $\curC$ is equal to the construction in Type 4, and $\curD$ cannot contain $\mu$ because the only codewords that could map to it under the surjective map $\curC\to\curD$ are $\mu$ and $\mu\cup\alpha$. However, we must be able to recover both $\mu$ and $\mu\cup \alpha$ as intersections of codewords in $\curC$. Using similar arguments to cases 2 and 3, we see that this implies $\mu$ is the intersection of codewords in $(\curD\cap \curI)\setminus\{\mu\}$, and also that $\curD$ must have a codeword properly containing $\mu$ but not in $\curI$. 
\end{itemize}
Thus if $\curC$ covers $\curD$, then $\curC$ is indeed isomorphic to one of the codes from Definition \ref{def:all-covering}. This proves the result.
\end{proof}

\acks{We are grateful to the anonymous referees for their thoughts and feedback on our paper. The authors gratefully acknowledge Zvi Rosen for providing valuable feedback during the early stages of this work. We also thank Alex Kunin for elucidating a compelling connection between morphisms of neural codes and Boolean matrix factorizations, as well as for highlighting a motivating application in connectomics. Finally, we greatly appreciate the insightful discussions with Carina Curto, which have inspired further development of this research.}

\nocite*{}
\bibliography{references}

\appendix

\section{Downward covering relations in $\pcode$: a revision}
\label{sec:down-covering}
This section is a slightly adjusted version of Section 3 in \citep{sunflowers19} to account for our assumption that the $\varnothing$-codeword exists in any neural code. We will characterize the downward covering relation in $\pcode$ in the sense that we combinatorially describe all codes covered by a given neural code $\curC$ in $\pcode$.

\begin{definition}[Trivial and redundant neurons]
\label{def:trivial-redundant-neurons}
Let $\curC\subseteq 2^{[n]}$ be a neural code. A neuron $i\in[n]$ is called \textbf{trivial} if $\Tk_\curC(i) = \varnothing$. A neuron $i$ is called \textbf{redundant} if there exists $\sigma\subseteq [n]\setminus \{i\}$ such that $\Tk_\curC(i) = \Tk_\curC(\sigma)$. 
\end{definition}

\begin{definition}[Covered neural codes]
\label{def:covered-codes}
Let $\curC\subseteq 2^{[n]}$ and for $j\in[n]$ let $T_j = \Tk_\curC(j)$. If $i\in[n]$ is a nontrivial neuron, then the \textbf{$i$th covered code} of $\curC$ is the image of $\curC$ under the morphism determined by the following collection of trunks: $\{T_j \mid T_j\neq T_i\}\cup \{T_j\cap T_i\mid T_j\cap T_i\neq T_i\}.$ This code is denoted $\curC^{(i)}$.
\end{definition}

\begin{example} 
\label{eg:covered-codes}
Let $\curC=\{\emptyset, 1, 12, 23, 13, 123\}$. Then the $1$st covered code \[
\curC^{(1)}=\{\emptyset,ac,ab,bd,abcd\}
\]
is the image of $\curC$ under the morphism defined by the following relabeled trunks
\[
T_a=\Tk_\curC(2),T_b=\Tk_\curC(3),T_c=\Tk_\curC(12),T_d=\Tk_\curC(13)
\]
as in \definitionref{def:defining-morphisms}. To verify, $\curC$ has 8 trunks, namely
\[
\Tk_\curC(\emptyset),\Tk_\curC(1),\Tk_\curC(2),\Tk_\curC(3),\Tk_\curC(12),\Tk_\curC(13),\Tk_\curC(23),\mbox{ and }\Tk_\curC(123),
\]
and $\curC^{(1)}$ has 7 trunks, namely
\[
\Tk_{\curC^{(1)}}(\emptyset),\Tk_{\curC^{(1)}}(a),\Tk_{\curC^{(1)}}(b),\Tk_{\curC^{(1)}}(c),\Tk_{\curC^{(1)}}(d),\Tk_{\curC^{(1)}}(ab),\Tk_{\curC^{(1)}}(ad)
\]
where $\Tk_{\curC^{(1)}}(c)=\Tk_{\curC^{(1)}}(ac)$, $\Tk_{\curC^{(1)}}(d)=\Tk_{\curC^{(1)}}(bd)$, and all other $\Tk_{\curC^{(1)}}(\sigma)$ are equal when $\sigma\in\{ad, bc, cd, abc, abd, bcd, abcd\}$
\end{example}

We will prove that all codes defined in \definitionref{def:covered-codes} are exactly the ones that $\curC$ covers, as long as the neuron $i$ is non-trivial and non-redundant.

\begin{definition}[Trunk generation]
\label{def:trunk-gen}
Let $\curC$ be a neural code and $\{T_1,\ldots, T_m\}$ a collection of trunks in $\curC$. We say that a trunk $T$ in $\curC$ is \textbf{generated} by $\{T_1,\ldots, T_m\}$ if there exists $\sigma\subseteq [m]$ such that $T = \bigcap_{i\in \sigma} T_i$. 
\end{definition}

By convention, we say that the empty intersection is all of $\curC$. Note that every nonempty trunk is generated by the set of simple trunks. Furthermore, a neuron $i$ is redundant if and only if $\Tk_\curC(i)$ is generated by $\{\Tk_\curC(j)\mid j\neq i\}$ by \definitionref{def:trivial-redundant-neurons} and \definitionref{def:trunk-gen}. 

\begin{lemma}\label{lem:factor}
Let $\curC, \curD,$ and $\curE$ be neural codes, and let $f:\curC\to\curD$ and $g:\curC\to\curE$ be surjective morphisms determined by collections of trunks $A$ and $B$, respectively. There exists a surjective morphism $h:\curD\to \curE$ such that the diagram below commutes (i.e. $g = h\circ f$) if and only if every trunk in $B$ is generated by $A$.
\[
\begin{tikzcd}
\curC \arrow[r, twoheadrightarrow, "f"] \arrow[rd, twoheadrightarrow, "g",swap] & \curD\arrow[d,twoheadrightarrow,dashed,"h"]\\
& \curE
\end{tikzcd}
\] 
\end{lemma}
\begin{proof}
See \cite[Lemma 3.13]{sunflowers19}.
\end{proof}

\begin{corollary}\label{cor:samenumbertrunks}
Let $f:\curC\to\curD$ be a surjective morphism. Then $f$ is an isomorphism if and only if $\curC$ and $\curD$ have the same number of trunks.
\end{corollary}
\begin{proof}
If $f$ is an isomorphism, then the inverse $f^{-1}$ yields a bijection on trunks in $\curC$ and trunks in $\curD$. The converse is exactly \cite[Proposition 3.16]{sunflowers19}.\end{proof}

\begin{corollary}\label{cor:deleteredundant}
Let $\curC\subseteq 2^{[n]}$ be a code, and suppose that $i\in[n]$ is a redundant neuron. Let $\curD = \{c\setminus \{i\}\mid c\in \curC\}$. The map $f:\curC\to\curD$ given by $f(c) = c\setminus \{i\}$ is an isomorphism.
\end{corollary}
\begin{proof}
The map described is a morphism since the preimage of $\Tk_\curD(\sigma)$ is simply $\Tk_\curC(\sigma)$. The codes $\curC$ and $\curD$ have the same number of trunks since every trunk in $\curC$ can be expressed as an intersection of simple trunks other than $\Tk_\curC(i)$, so $f$ is an isomorphism.
\end{proof}\begin{proposition}\label{prop:intersectionpreimage}
Let $f:\curC\to\curD$ be a morphism. If $S$ and $T$ are trunks in $\curD$, then $f^{-1}(S\cap T) = f^{-1}(S)\cap f^{-1}(T)$.
\end{proposition}
\begin{proof} This is true for any function $f:\curC\to\curD$.\end{proof}
\begin{lemma}\label{lem:Cifactor}
Let $\curC\subseteq [n]$ and let $f:\curC\to\curD$ be a surjective morphism, and suppose that $f$ is not an isomorphism. Then there exists a neuron $i\in[n]$ so that $\curD\le \curC^{(i)}$. 
\end{lemma}
\begin{proof}
By \cite[Proposition 3.15]{sunflowers19}, the map $f^{-1}$ is an injection from trunks in $\curD$ to those in $\curC$. \corollaryref{cor:samenumbertrunks} implies that $\curC$ has more trunks than $\curD$, and so there must be some trunk in $\curC$ that is not the preimage of a trunk in $\curD$.  In fact there must be a simple trunk in $\curC$ that is not the preimage of a trunk in $\curD$ (Proposition \ref{prop:intersectionpreimage} implies that the set of trunks in $\curC$ which are preimages of trunks in $\curD$ is closed under intersection, and every nonempty trunk is an intersection of simple trunks in $\curC$). Let $i\in[n]$ be a neuron so that $\Tk_\curC(i)$ is not the preimage of a trunk in $\curD$.

Recall that the surjective morphism $\curC\to\curC^{(i)}$ is determined by a set of trunks which generate all nonempty trunks in $\curC$ except for possibly $\Tk_\curC(i)$. Thus by \lemmaref{lem:factor} there exists a surjective map $g:\curC^{(i)}\to\curD$ whose composition with $\curC\to\curC^{(i)}$ is equal to $f$. This shows that $\curD\le \curC^{(i)}$ as desired. 
\end{proof}

\begin{theorem}\label{thm:covering-above}
Let $\curC\subseteq 2^{[n]}$ and $\curD\subseteq 2^{[m]}$ be codes. Then $\curC$ covers $\curD$ in $\pcode$ if and only if $\curD\cong \curC^{(i)}$ for some non-redundant, nontrivial neuron $i\in[n]$. 
\end{theorem}
\begin{proof}
By \corollaryref{cor:deleteredundant} we can reduce to the case where $\curC$ has no redundant neurons. Then suppose that $\curC$ covers $\curD$, so there exists a surjective morphism $f:\curC\to\curD$ that is not an isomorphism. \lemmaref{lem:Cifactor} implies that for some $i\in[n]$ we have $\curD\le \curC^{(i)}$. But $\curC^{(i)} < \curC$ since $i$ is not redundant, and since $\curC$ covers $\curD$ we must have that $\curD = \curC^{(i)}$ in $\pcode$ so they are isomorphic.

For the converse, we must argue that $\curC$ covers $\curC^{(i)}$. The two codes are not isomorphic since $\curC^{(i)}$ has exactly one less trunk than $\curC$. Any code $\curD$ such that $\curC^{(i)} \le \curD \le \curC$, either has the same number of trunks as $\curC$ or as $\curC^{(i)}$, and is isomorphic to one of them by \corollaryref{cor:samenumbertrunks}, so $\curC$ covers $\curC^{(i)}$ as desired.  
\end{proof}

\begin{corollary}[Covering criterion]
\label{cor:covering-criterion}
Let $f:\curC\to\curD$ be a surjective morphism. Then $\curC$ covers $\curD$ if and only if $\curC$ has exactly one more trunk than $\curD$. 
\end{corollary}
\begin{proof}
Suppose that $\curC$ covers $\curD$. By \theoremref{thm:covering-above}, $\curD\cong \curC^{(i)}$ for some non-redundant $i\in[n]$. But $\curC^{(i)}$ has exactly one less trunk than $\curC$, since the surjective morphism $\curC\to \curC^{(i)}$ is determined by a set of trunks which generates all trunks in $\curC$ except for $\Tk_\curC(i)$.

For the converse, suppose that $\curC$ has exactly one more trunk than $\curD$. For any code $\curE$ with $\curC \ge \curE \ge \curD$, we see that $\curE$ has as many trunks as either $\curC$ or $\curD$ depending on which inequality is strict, and \corollaryref{cor:samenumbertrunks} implies that $\curE$ is isomorphic to one of $\curC$ or $\curD$. Thus $\curC$ covers $\curD$ as desired.
\end{proof}

\section{Intersection-complete neural codes}
\label{sec:int-comp-up-covering}
This section is supplementary to \sectionref{sec:up-covering}. We will focus only on intersection-complete codes. It turns out that describing all codes that cover a given neural code $\curC$ in $\pcode$ is a simpler task if we assume that $\curC$ is intersection-complete and restrict our attention to the intersection-complete codes that cover it. We will use these results in \sectionref{sec:up-covering} to describe the general case. Isolated subsets defined in \definitionref{def:isolated-subcode} are the key objects that we will use to describe the covering relation in $\pcode$ for intersection-complete codes. Our main construction is as follows.

\begin{definition}\label{def:int-comp-covering}
Let $\curC$ be any intersection-complete code and let $\curI\subseteq \curC$ be an isolated subset with minimal element $\mu$. Define 
\begin{equation}
\curC_{[\curI]} = \, \{\mu\}\, \cup\, \curC\setminus\, \curI\, \cup\, (\curI)_\alpha
\end{equation}
where $(\cdot)_\alpha$ is defined the same as in \definitionref{def:all-covering}. In other words, we obtain $\curC_{[\curI]}$ by adding a new neuron $\alpha$ to every codeword in $\curI$ while also maintaining $\mu$ as a codeword.
\end{definition}

When $\curI = \{\mu\}$, the code $\curC_{[\curI]}$ is the result of adding $(\mu)_\alpha$ as a new codeword. Furthermore, we can see that $\curC\leq\curC_{[\curI]}$, since there is a surjective morphism from $\curC_{[\curI]}$ to $\curC$ given by the deletion of the neuron $\alpha$. Our goal is to prove that the construction in \definitionref{def:int-comp-covering} is the only way to construct intersection-complete covering codes of a given intersection-complete code. We start with some supporting propositions and lemmas.

\begin{proposition}
\label{prop:trunks-elems-bijection}
A code $\curC$ is intersection-complete if and only if the assignment $\sigma\mapsto \Tk_\curC(\sigma)$ defines a bijection between $\curC$ and its nonempty trunks.
\end{proposition}
\begin{proof} Assume the map $\sigma\mapsto \Tk_\curC(\sigma)$ is a bijection between $\curC$ and its nonempty trunks. Let $\sigma_1,\sigma_2\in \curC$, and $T$ an inclusion-minimal trunk that contains both codewords. Then there is $\sigma$ such that $T = \Tk_\curC(\sigma)$. Clearly, $c\subseteq \sigma_1\cap\sigma_2$ since $\sigma$ is contained in both $\sigma_1$ and $\sigma_2$. If $\sigma$ were a proper subset of $\sigma_1\cap\sigma_2$, then $\Tk_\curC(\sigma_1\cap \sigma_2)$ would be properly contained in $T$, which is a contradiction to the choice of $T$. Thus, $\sigma = \sigma_1\cap \sigma_2$.

Conversely, suppose $\curC$ is intersection-complete. The assignment $\sigma\mapsto \Tk_\curC(\sigma)$ clearly defines an injection because $\sigma$ is the unique minimal element of its trunk in $\curC$. The map is a surjection because if $T$ is a nonempty trunk then the intersection of all elements of $T$ must be a codeword $\sigma$ such that $T = \Tk_\curC(\sigma)$.
\end{proof}

This proposition gives us a quick way to compare the numbers of trunks between any two intersection-complete codes based on the numbers of their codewords.

\begin{proposition}
\label{prop:int-comp-wo-alpha}
Let $\curC$ be an intersection-complete code and $\curI\subseteq \curC$ an isolated subset of $\curC$ with minimal element $\mu$. Then, $\{\mu\}\cup(\curC\setminus \curI)$ is intersection-complete.
\end{proposition}
\begin{proof}
\definitionref{def:isolated-subcode} implies that the intersection of codewords in $\curC\setminus\curI$ is an element of $\curC\setminus\curI$ or equal to $\mu$. Intersecting an element of $\curC\setminus\curI$ with $\mu$ itself yields either $\mu$ or some codeword in $\curC$ that is below $\mu$ which is definitely not in $\curI$. Thus $\{\mu\}\cup (\curC\setminus\curI)$ is intersection-complete.
\end{proof}

\begin{proposition}
\label{prop:int-comp-w-alpha}
Let $\curC$ be an intersection-complete code and $\curI\subseteq\curC$ an isolated subset. The code $\curC_{[\curI]}$ of \definitionref{def:int-comp-covering} is intersection-complete. 
\end{proposition}
\begin{proof}
Let $\mu$ be the minimal element of $\curI$. The codewords in $\curC_{[\curI]}$ come in two types: those in $\{\mu\}\cup(\curC\setminus \curI)$ and those of the form $c\cup\alpha$ where $c\in\curI$. For any two codewords $c_1,c_2\in \curC_{[\curI]}$, there are 3 cases to consider.

\underline{Case 1:} If both $c_1$ and $c_2$ lie in $\{\mu\}\cup(\curC\setminus \curI)$, then so does $c_1\cap c_2$ by Proposition \ref{prop:int-comp-wo-alpha}. This guarantees $c_1\cap c_2\in \curC_{[\curI]}$ in this case.

\underline{Case 2:} If $c_1=a_1\cup\alpha$ and $c_2=a_2\cup\alpha$ for some $a_1,a_2\in \curI$, then $c_1\cap c_2 = (a_1\cap a_2)\cup \{\alpha\}$. Since $\curI$ is intersection-complete, $a_1\cap a_2\in \curI$. This means $c_1\cap c_2$ is in $(\curI)_\alpha$ and a codeword of $\curC_{[\curI]}$.

\underline{Case 3:} Without loss of generality, suppose that $c_1$ contains $\alpha$ and $c_2$ does not. In this case $c_1 = a_1\cup \alpha$ for some $a_1\in\curI$, and $c_2\in \curC\setminus \curI$. We can immediately see that $c_1\cap c_2 = a_1\cap c_2$. By \definitionref{def:isolated-subcode}, we know that $a_1\cap c_2$ is in $\{\mu\}\cup (\curC\setminus \curI)$ and hence a codeword of $\curC_{[\curI]}$.

In all cases $c_1\cap c_2$ lies in $\curC_{[\curI]}$, and thus this code is intersection-complete as desired. \end{proof}

\begin{proposition}\label{prop:CI-covers-C}
Let $\curC$ be an intersection-complete code and $\curI\subseteq \curC$ an isolated subset. Then $\curC_{[\curI]}$ covers $\curC$ in $\pcode$.
\end{proposition}
\begin{proof}
There is a natural surjective morphism $\curC_{[\curI]}\to \curC$ given by deleting the neuron $\alpha$. By \corollaryref{cor:covering-criterion}, it suffices to prove that $\curC_{[\curI]}$ has exactly one more nonempty trunk than $\curC$. Since $\curC$ is intersection-complete, $\curC_{[\curI]}$ is intersection-complete by Proposition \ref{prop:int-comp-w-alpha}. Note that nonempty trunks in an intersection-complete code are in bijection with its codewords. \definitionref{def:int-comp-covering} shows that $\curC_{[\curI]}$ has one more codeword than $\curC$, so the result follows.
\end{proof}

\begin{lemma}
\label{lem:almost-bijection}
Let $\curC\subseteq 2^{[n]}$ be an intersection-complete code, and $i\in[n]$ a non-redundant, nontrivial neuron. Let $\mu$ be the minimal element of $\Tk_\curC(i)$, and let $f:\curC\to\curC^{(i)}$ be the natural surjective morphism. Then the following are true:
\begin{itemize}
\item[(i)] $\mu\setminus\{i\}$ is also a codeword in $\curC$, and 
\item[(ii)] $f(\mu) = f(\mu\setminus \{i\})$. 
\end{itemize}
\end{lemma}
\begin{proof}
For (i), suppose for contradiction that $\mu\setminus \{i\}$ was not a codeword of $\curC$. Then the intersection of all codewords in $\Tk_\curC(\mu\setminus \{i\})$ must properly contain $\mu\setminus \{i\}$. If this intersection is not $\mu$, then intersecting it with $\mu$ will yield $\mu\setminus \{i\}$ as a codeword because $\curC$ is intersection-complete. Otherwise, the minimal element of $\Tk_\curC(\mu\setminus \{i\})$ is $\mu$, which implies that $\Tk_\curC(i) = \Tk_\curC(\mu\setminus \{i\})$. This also contradicts the assumption that $i$ is not a redundant neuron.

For (ii), recall from \definitionref{def:covered-codes} that $f$ is determined by the collection of trunks $\{T_j \mid T_j\neq T_i\}\cup \{T_j\cap T_i\mid T_j\cap T_i\neq T_i\}$ where $T_j = \Tk_\curC(j)$. Since $\mu$ is the unique minimal element of $T_i$, trunks of the form $T_j\cap T_i$ above do not contain it (otherwise they would be the same as $T_i$). Thus $f(\mu)$ records which $T_j$ contains $\mu$, except for $T_i$. In other words, $f(\mu)$ records the neurons in $\mu$ other than neuron $i$. These are exactly the neurons that are in $\mu\setminus \{i\}$, so $f(\mu) = f(\mu\setminus\{i\})$. 
\end{proof}

\begin{proposition}\label{prop:image-of-intersection}
Let $f:\curC\to \curD$ be a morphism, and let $c_1,c_2\in \curC$ be such that $c_1\cap c_2\in\curC$. Then $f(c_1\cap c_2)=f(c_1)\cap f(c_2)$.
\end{proposition}
\begin{proof}
We know that $f$ is determined by some collection of trunks $\{T_1,\ldots, T_m\}$ in $\curC$. But this implies that \[
f(c_1\cap c_2) = \{j\in [m]\mid c_1\cap c_2\in T_j\} = \{j\in[m]\mid c_1\in T_j\text{ and } c_2\in T_j\} = f(c_1)\cap f(c_2). 
\]
\end{proof}

\begin{lemma}\label{lem:image-of-Ti}
Let $\curC\subseteq 2^{[n]}$ be an intersection-complete code, and $i\in [n]$ be a non-redundant, nontrivial neuron. Let $f:\curC\to\curC^{(i)}$ be the natural surjective morphism. Then $\curC^{(i)}$ is intersection-complete, and $\curI = f(\Tk_\curC(i))$ is an isolated subset in $\curC^{(i)}$.
\end{lemma}
\begin{proof}
It is known that the image of an intersection-complete code is again intersection-complete (see \cite[Theorem 1.5]{morphisms20}), and so $\curC^{(i)}$ is intersection-complete. The set $\Tk_\curC(i)$ is certainly intersection-complete, and since $\curI$ is the image of this set under $f$, it follows that $\curI$ is intersection-complete. To see that $\curI$ is isolated in $\curC^{(i)}$, it remains to show that it satisfies \definitionref{def:isolated-subcode}.   

 Let $c_1\in \curC^{(i)}\setminus \curI$ and $c_2\in\mathcal{I}\setminus\{\mu'\}$ where $\mu'$ is the minimal element of $\curI$. Note that $\mu'=f(\mu)=f(\mu\setminus\{i\})$ where $\mu$ is the minimal element of $\Tk_\curC(i)$ by part (ii) of \lemmaref{lem:almost-bijection}. Suppose for contradiction that $c_1\supset c_2$, which implies $c_1\cap c_2\in\curI\setminus\{\mu'\}$. Let $a_1, a_2\in\curC$ such that $f(a_1)=c_1$ and $f(a_2)=c_2$. By Proposition \ref{prop:image-of-intersection}, $f(a_1\cap a_2)=c_1\cap c_2$. Because $c_1\notin\curI$, $a_1$ is not in $\Tk_\curC(i)$, and neither is $a_1\cap a_2$. Again, part (ii) of  \lemmaref{lem:almost-bijection} implies that $f$ is bijective except that it identifies $\mu$ and $\mu\setminus\{i\}$. Since $a_1\cap a_2$ is not in of $\Tk_\curC(i)$, this implies that if $c_1\cap c_2=f(a_1\cap a_2)$ is an element of $\curI = f(\Tk_\curC(i))$, then the only case that can happen is $a_1\cap a_2 = \mu\setminus \{i\}$. Then $c_1\cap c_2$ must be $\mu'$, which is a contradiction.


\end{proof}

We are now ready to prove the following theorem.

\begin{theorem}
\label{thm:int-comp-covering}
Let $\curC\subseteq 2^{[n]}$ and $\curD\subseteq 2^{[m]}$ be intersection-complete codes. The following are equivalent:
\begin{itemize}
\item[(i)] $\curC$ covers $\curD$ in $\pcode$, and
\item[(ii)] $\curC\cong \curD_{[\curI]}$ where $\curI\subseteq\curD$ is some isolated subset.
\end{itemize}
\end{theorem}

\begin{proof}
The fact that (ii) implies (i) is proven in Proposition \ref{prop:CI-covers-C}. For the converse, suppose that $\curC$ covers $\curD$. From \theoremref{thm:covering-above}, we know that $\curD\cong \curC^{(i)}$ for some non-redundant, nontrivial neuron $i$. Since the choice of isolated subset $\curI\subseteq \curD$ is invariant up to isomorphism, it suffices to prove the result when $\curD = \curC^{(i)}\subseteq 2^{[m]}$. 

Let $f:\curC\to \curC^{(i)}$ be the natural surjective morphism, $T_j = \Tk_\curC(j)$ for $j\in [n]$, and $\curI = f(T_i)$. By \lemmaref{lem:image-of-Ti}, $\curI$ is an isolated subset of $\curC^{(i)}$ and so we may define $\curE = (\curC^{(i)})_{[\curI]}\subseteq 2^{[m+1]}$ with the new neuron $\alpha=m+1$. Let $g:\curE\to \curC^{(i)}$ be the natural surjective morphism defined by deleting the new neuron $m+1$. Lastly, let $\mu'$ be the minimal element of $\curI$ in $\curC^{(i)}$, and $\mu$ the minimal element of $T_i$ in $\curC$. We will prove that $\curE\cong\curC$ by constructing $h:\curC\to\curE$ as follows:
\[
h(c) = \begin{cases}
f(c) & c\notin T_i,\\
f(c)\cup \{m+1\} & c\in T_i. 
\end{cases}
\]
Observe that $g(h(c)) = f(c)$ for all $c$. Moreover, observe that $h$ is bijective since $\curE$ is constructed from $\curC^{(i)}$ by adding $m+1$ to all codewords in $f(T_i)$ while maintaining $\mu'$ as a codeword, and $\mu' = f(\mu\setminus \{i\})$ by (ii) of \lemmaref{lem:almost-bijection}. 

\[
\begin{tikzcd}
\curC \arrow[r, rightarrow, "h"] \arrow[rd, twoheadrightarrow, "f",swap] & \curE=(\curC^{(i)})_{[\curI]}\arrow[d,twoheadrightarrow,"g"]\\
& \curC^{(i)}
\end{tikzcd}
\] 
To prove that $h$ is a morphism, observe that it is the morphism determined by the collection of trunks $\{T_j \mid T_j\neq T_i\}\cup \{T_j\cap T_i\mid T_j\cap T_i\neq T_i\}$ together with a new trunk $\Tk_\curC(i)$ corresponding to the new neuron $m+1$. Since $h$ is a bijective morphism between two intersection-complete codes $\curC$ and $\curE$, they must have the same number of trunks and hence be isomorphic.
\end{proof}
\end{document}